\documentclass[11pt,a4paper,twoside]{article}
\usepackage{amsmath, amssymb, amsthm} 

 \newtheorem{theorem}{Theorem}
 \newtheorem{corollary}{Corollary}
 \newtheorem{lemma}{Lemma}
 \newtheorem{proposition}{Proposition}
 \newtheorem{definition}{Definition}
 \newtheorem{remark}{Remark}
 \newtheorem{example}{Example}

\def\calf{{\mathcal F}}
\def\vol{\operatorname{vol}}
\newcommand\id{\operatorname{id}}
\def\RR{\mathbb{R}}
\def\<{\langle}
\def\>{\rangle}
\newcommand\tr{\operatorname{tr}}
\newcommand\Div{\operatorname{div}}
\def\Ric{\operatorname{Ric}}

\newcommand{\eq}{\hspace*{-1.4mm}&=&\hspace*{-1.4mm}}

\author{Vladimir Rovenski\footnote{Department of Mathematics, University of Haifa, Mount Carmel, 3498838 Haifa, Israel
       \newline e-mail: {\tt vrovenski@univ.haifa.ac.il} } }

\title{Integral formulas for a foliated sub-Riemannian manifold}

\begin{document}

\date{}

\maketitle

\begin{abstract}

In this article, we deduce a series of integral formulas for a foliated sub-Riemannian manifold,
which is a new geometric concept denoting a Riemannian manifold equipped with a distribution ${\mathcal D}$ and a foliation ${\mathcal F}$,
whose tangent bundle is a subbundle of ${\mathcal D}$.
Our integral formulas generalize some results for foliated Riemannian manifolds
and involve the shape operators of ${\mathcal F}$ with respect to normals in ${\mathcal D}$ and
the curvature tensor of induced connection on ${\mathcal D}$.
The formulas also include arbitrary functions $f_j\ (0\le j<\dim{\mathcal F})$ depending on scalar invariants of the shape operators,
and for a special choice of $f_j$ reduce to integral formulas with the Newton transformations of the shape operators.
We apply our formulas to foliated sub-Riemannian manifolds with restrictions on the curvature and extrinsic geometry of~${\mathcal F}$
and to codimension-one foliations.



\end{abstract}

\maketitle

\section*{Introduction}

A non-holonomic manifold is a pair $(M,{\mathcal D})$, where $M$ is a smooth manifold and
${\mathcal D}$ is a distribution on $M$, i.e., a subbundle of the tangent bundle $TM$, see~\cite{BF}.
In terms of an almost product structure on $M$, i.e., a (1,1)-tensor field $P$ of constant rank for which $P^2=P$, see \cite{g1967},
we obtain
${\mathcal D}=P(TM)$.
A non-holonomic manifold $(M,{\mathcal D})$ with a Riemannian metric $g$ on the distribution ${\mathcal D}$,
e.g., the restriction of a metric on the manifold $M$, is the main object of sub-Riemannian geometry.
In \cite{Bau}, they explore sub-Riemannian structures arising as the transversal distribution to a foliation.
 In \cite{r-subR1}, we introduced a new geometrical concept of a \textit{foliated sub-Riemannian manifold},
i.e., $(M,{\mathcal D},g)$ equipped with
a foliation $\calf$, whose tangent bundle $T\calf$ is a subbundle of any codimen\-sion of~${\mathcal D}$.
(Recall that a foliation $\calf$
is a partition of $M$ into equivalence classes that locally models the partition of $\RR^{n+m}$ by submanifolds parallel to $\RR^n$).
We~proved in \cite{r-subR1} a series of integral formulas for
a codimension-one foliated sub-Riemannian manifold,
i.e., ${\mathcal D}=T\calf\oplus\,{\rm span}(N)$, where $N$ is a unit vector field orthogonal to $\calf$;
these integral formulas generalize results in \cite{aw2010,blr} for
foliated Riemannian manifolds.
On the other hand, in \cite{r8}, we extended the above mentioned results in \cite{aw2010,blr}
for foliations of arbitrary codimension of Riemannian mani\-folds,
as a series of integral formulas depending on the Newton or even more general transformations of the shape operators of the leaves.

Analyzing history of extrinsic geometry of foliations, we see that from the origin it was related to some integral formulas containing the shape operator
(or the second fundamental form) of leaves and its invariants (mean curvature, higher order mean curvatures $\sigma_r$, etc.) and some expressions corresponding to geometry (curvature) of $M$.
Integral formulas are useful for solving many problems in differential geometry, both manifolds (for example, the Gauss-Bonnet formula for closed surfaces) and foliations, see surveys \cite{arw,RWa-1,Rov-Wa-2021}.
The~first known integral formula for a codimension-one foliation of a closed Riemannian manifold by G.\,Reeb \cite{reeb1}
tells that the integral of the mean curvature $H$ of the leaves is equal to zero.
Its proof is based
on the application of the Divergence theorem to the identity $\Div N=-H$ with $N$ a unit normal to the leaves.
The~second formula in the series of total higher mean curvatures $\sigma_r$'s of a codimension-one foliation is, e.g.,~\cite{r8},
\begin{equation}\label{E-sigma2}
 \int_M (2\,\sigma_2-\Ric_{N,N})\,{\rm d}\vol_g=0.
\end{equation}
In \cite{aw2010}, the Newton transformations $T_r(A)$ (of the shape operator $A$ of the leaves) were applied to codimension one foliations,
and a series of integral formulas for $r\ge0$ starting with \eqref{E-sigma2} was obtained (with consequences for foliated space forms,
see~\cite{blr}):
\begin{equation}\label{E-intNTp1}
 \int_{M}\big((r+2)\,\sigma_{r+2}-\tr_{\,\calf}(T_r(A) {\mathcal R}_{N})
 -\<{\rm div}_\calf\, T_{r}(A),\,\nabla_{N} N\>
 \big)\,{\rm d}\vol_g=0.
\end{equation}
Here, the Jacobi operator ${\mathcal R}_{N}:T\calf\to T\calf$ is given by ${\mathcal R}_{\,N}(Y)=R(Y, N)N$.
There is also a series of integral formulas for two complementary orthogonal distributions ${\cal D}$ and ${\cal D}^\bot$ (or foliations) of arbitrary dimension on a closed Riemannian manifold $(M,g)$,
see \cite{r8}, starting with the
the following generalization of \eqref{E-sigma2}, see \cite{Wa01}:
\begin{equation}\label{E-PW-int}
 \int_M\big({\rm S\,}_{\rm mix} +\|h\|^2+\|h^\bot\|^2-\|H\|^2-\|H^\bot\|^2-\|T\|^2-\|T^\bot\|^2\big)\,{\rm d}\vol_g=0.
\end{equation}
Here, $h,T: {\mathcal D}\times {\mathcal D}\to {\mathcal D}^\bot$ and $H=\tr_g h$
are the second fundamental form,
the integrability tensor and the mean curvature vector field of
${\mathcal D}$, and similarly for ${\mathcal D}^\bot$,
${\rm S\,}_{\rm mix}$ is the mixed scalar curvature.
 The~following question naturally arises:
\textit{if the integral formulas of the type \eqref{E-intNTp1} and \eqref{E-PW-int}
can be proved for foliated sub-Riemannian manifolds}?

In this article, we
answer this question in the affirmative and derive a series of integral formulas
(in Theorems~\ref{T-mainLW} and \ref{T-mainLW-2} and Corollaries~\ref{C-main00}--\ref{T-mainLW-1})
for a foliated sub-Riemanni\-an mani\-fold,
which generalize some known integral formulas for foliated Riemannian manifolds, see \eqref{E-intNTp1} and \cite{aw2010,bn},
and extend the results in \cite{r-subR1}, where codim\,$\calf=1$.
Our integral formulas involve the shape operators of $\calf$ with respect to unit normals in ${\mathcal D}$,
some components of the curvature tensor of the induced connection on ${\mathcal D}$.
The formulas also include arbitrary functions $f_j\ (0\le j<\dim\calf)$ depending on scalar invariants of the shape operators,
and for a special choice $f_j=(-1)^j\sigma_{r-j}$ reduce to integral formulas with Newton transformations of the shape operators
(and their scalar invariants -- the $r$th mean curvatures) of $\calf$.
We apply our formulas to foliated sub-Riemannian manifolds with restrictions on the curvature and extrinsic geometry of the leaves
of $\calf$ and to codimension-one foliations.

\section{Preliminaries}

Here, we define the shape operator $A_\xi$ and its Newton transformations $T_r(A_\xi)$,
the $r$th mean curvatures $\sigma_r(\xi)$ and power sums symmetric functions $\tau_k(\xi)$.

Let ${\mathcal D}$ be an $(n+p)$-dimensional distribution on a smooth $m$-dimensional manifold $M$, i.e.,
a subbundle of $TM$ of rank $n+p$ (where $n,p>0$ and $n+p<m$).
In other words, to each point $x\in M$ we assign an $(n+p)$-dimensional subspace ${\mathcal D}_x$ of the tangent space $T_xM$ smoothly depending on $x$.
An integrable distribution determines a foliation; in this case, the Lie bracket of any two vector fields from the distribution ${\mathcal D}$ also belongs to
 ${\mathcal D}$.

A pair $(M,{\mathcal D})$, where ${\mathcal D}$ is a non-integrable distribution on a manifold $M$, is called a \textit{non-holonomic manifold}, see~\cite{BF}.
The concept of a non-holonomic manifold was introduced for the geometric interpretation of constrained systems in classical mechanics.
A~\textit{sub-Riemannian manifold} is a non-holonomic manifold $(M,{\mathcal D})$, where ${\mathcal D}$ is equipped with a sub-Riemannian metric $g=\<\cdot\,,\cdot\>$, i.e.,
the scalar product on ${\mathcal D}_x$ for all $x\in M$, see \cite{BF}.
Usually, they assume that the sub-Riemannian metric on ${\mathcal D}$ (the horizontal bundle)
is extended to a Riemannian metric on the whole $M$, also denoted by $g$.
This allows us to define the orthogonal distribution $\widetilde{\mathcal D}$ (the vertical subbundle) such that $TM={\mathcal D}\oplus\widetilde{\mathcal D}$.

\begin{definition}\rm
A sub-Riemannian manifold $(M,{\mathcal D},g)$ equipped with a foliation $\calf$ such that the tangent bundle $T\calf$ is a subbundle of ${\mathcal D}$ will be called a \textit{foliated sub-Riemannian manifold}.
\end{definition}

Let $\calf$ be a foliation of codimension $p$ relative to ${\mathcal D}$, i.e., $\dim\calf=n$,
and let ${\mathcal N}\calf$ be the orthogonal complement of $T\calf$ in~${\mathcal D}$; thus, the following decomposition holds:
 ${\mathcal D}=T\calf\oplus {\mathcal N}\calf$.
 Let $^\top$ denotes the projection from $TM$ on the vector subbundle $T\calf$.
The \textit{shape operator} $A_\xi: T\calf\to T\calf$ of the foliation $\calf$ with respect to the unit normal $\xi\in {\mathcal N}\calf$ is defined by
\begin{equation*}
 A_\xi(X)=-(\nabla_X\,\xi)^\top,\quad X\in T\calf.
\end{equation*}
The \textit{elementary symmetric functions} $\sigma_j(A_\xi)$ of $A_\xi$
are given~by the equality, e.g., \cite{RWa-1},
\[
 \sum\nolimits_{\,r=0}^{\,n}\sigma_r(A_\xi)\,t^r=\det(\,\id_{\,T\calf}+\,t\,A_\xi),\quad t\in\RR.
\]
Note that
 $\sigma_r(A_\xi) = \sum\nolimits_{\,i_1<\cdots<i_r} \lambda_{i_1}(\xi)\cdots \lambda_{i_r}(\xi)$,
where $\lambda_1(\xi)\le\ldots\le\lambda_n(\xi)$ are the eigenvalues of $A_\xi$.
The \textit{power sums} symmetric functions of $A_\xi$ are
\[
 \tau_j(A_\xi)=\tr\,(A_\xi^j)=\sum\nolimits_{\,1\le i\le n}(\lambda_i(\xi))^j,\quad j\in\mathbb{N}.
\]
For short, set
\[
 \sigma_r(\xi)=\sigma_{r}(A_\xi),\quad
 \tau_{r}(\xi)=\tau_{r}(A_\xi).
\]
For~example,
$\sigma_0(\xi)=1$, $\sigma_n(\xi)=\det A_\xi$,
$\sigma_1(\xi)=\tau_1(\xi)=\tr A_\xi$ and $2\,\sigma_2(\xi)=\tau_1^2(\xi)-\tau_2(\xi)$.


\begin{definition}\rm
The \textit{Newton transformations} $T_{r}(A_\xi)$ of the shape operator $A_\xi$ of
a foliated sub-Riemannian manifold $(M,{\mathcal D},\calf,g)$ are defined recursively or explicitly by
\begin{equation}\label{E-defNTind}
  T_0(A_\xi) = \id_{\,T\calf},\quad T_r(A_\xi)=\sigma_r(\xi)\id_{\,T\calf} -A_\xi\,T_{r-1}(A_\xi)\quad (0< r\le n),
\end{equation}
\begin{equation}\label{E-intNT-C}
 T_r(A_\xi) = \sum\nolimits_{j=0}^r(-1)^j\sigma_{r-j}(\xi)\,A_\xi^j =\sigma_r(\xi)\id_{\,T\calf} -\sigma_{r-1}(\xi)\,A_\xi+\ldots+(-1)^r A_\xi^{\,r}.
\end{equation}
\end{definition}

For example, $T_1(A_\xi)=\sigma_1(\xi)\id_{\,T\calf} - A_\xi$ and $T_n(A_\xi)=0$.
Define a (1,1)-tensor field
\begin{equation}\label{E-barCN-1}
 {\mathcal A}_\xi:=\sum\nolimits_{\,j=0}^{\,n-1} f_j(\tau_1(\xi),\ldots,\tau_n(\xi))\,A_\xi^j,
\end{equation}
where $f_j:\RR^n\to\RR\ (0\le j< n)$ are given functions.
We write $f_j=f_j(\xi)$ shortly.
 The~choice of the RHS for ${\mathcal A}_\xi$
 in (\ref{E-barCN-1}) is natural for the following reasons, \cite[Section~1.2.2]{RWa-1}:

\smallskip\noindent\ \
$\bullet$ the powers $A_\xi^j$ are the only (1,1)-tensors, obtained algebraically from $A_\xi$,
while $\tau_1(\xi),$ $\ldots,\tau_n(\xi)$, or, equivalently,
$\sigma_1(\xi), \ldots, \sigma_n(\xi)$, generate all scalar invariants of $A_\xi$.

\smallskip\noindent \ \ \
$\bullet$ the Newton transformation $T_r(A_\xi)\ (0\le r<n)$ of (\ref{E-intNT-C}) depends on all $A_\xi^j\ (0\le j\le r)$.

\smallskip\noindent
In this article, for illustrate the results for ${\mathcal A}_\xi$ by
the particular case of ${\mathcal A}_\xi = T_r(A_\xi)$.

Let $\nabla^\calf:TM\times T\calf\to T\calf$ be the induced connection on the vector subbundle $T\calf$.
Since the (1,1)-tensors $A_\xi$ and $T_r(A_\xi)$ are self-adjoint, we have
\begin{equation*}
 \<(\nabla^\calf_{X}\,T_r(A_\xi))Y, \,V\> = \<(\nabla^\calf_{X}\,T_r(A_\xi))V, \,Y\>,\quad X,Y,V\in T\calf.
\end{equation*}
The following properties of $T_r(A)$ are proved similarly as for codimension-one foliations of a Riemannian manifold,
e.g., \cite{aw2010} or
\cite[Lemma~1.3]{RWa-1}.

\begin{lemma}
\label{L-NTprop}
For the shape operator $A_\xi$ of an $n$-dimensional foliation $\calf$ we have
 \begin{eqnarray*}
 \tr_{\,\calf} T_r(A_\xi)\eq(n-r)\,\sigma_r(\xi),\\
 \tr_{\,\calf} (A_\xi\cdot T_r(A_\xi)) \eq (r+1)\,\sigma_{r+1}(\xi),\\
 \tr_{\,\calf} (A^2\cdot T_r(A_\xi)) \eq \sigma_{1}(\xi)\,\sigma_{r+1}(\xi)-(r+2)\,\sigma_{r+2}(\xi),\\
 \tr_{\,\calf} ( T_{r-1}(A_\xi)(\nabla_X^\calf A_\xi)) \eq X(\sigma_{r}(\xi)),\quad X\in T\calf,\\
 k\/\tr_{\,\calf}(A^{k-1}_\xi\nabla^{\calf}_{X} A_\xi) \eq X(\tau_k(\xi)),\quad X\in T\calf,\ \ k>0.
 \end{eqnarray*}
\end{lemma}

\section{The induced connection and curvature on ${\mathcal D}$}

Here, we define the induced linear connection $\nabla^P$ and the curvature tensor $R^P$ related to a foliated sub-Riemannian manifold,
we also prove Codazzi type equation.
The orthoprojector $P:TM\to{\mathcal D}$  on the distribution ${\mathcal D}$
is characterized by the properties, e.g.,~\cite{g1967},
\begin{equation*}
 P=P^*\ (\textrm{self-adjoint}),\quad P^2=P.
\end{equation*}
The Levi-Civita connection $\nabla$ on $(M,g)$ induces a linear connection $\nabla^P$ on ${\mathcal D}$:
\[
 \nabla^P_X (PY) = P\nabla_X (PY),\quad X,Y\in\Gamma(TM),
\]
which is compatible with the metric on ${\cal D}$: $X\<U,V\> = \<\nabla^P_X U,\, V\> +\<U,\, \nabla^P_X V\>$ for
$U,V\in{\cal D}$.
Set ${R}^P(X,Y,U,V)=\<{R}^P(X,Y)U,\, V\>$ for $U,V\in{\mathcal D}$, where $R^P:TM\times TM\rightarrow{\rm End}({\mathcal D})$ 
is the curvature tensor of $\nabla^P$.
As for any linear connection, we have the anti-symmetry for the first pair of vectors: ${R}^P(Y,X)U=-{R}^P(X,Y)U$.
Since $\nabla^P$ is compatible with $g$, then the anti-symmetry for the last pair of vectors in ${\cal D}$ is valid, e.g., \cite{jj2011},
\begin{equation}\label{E-R-symm2}
 {R}^P(X,Y,U,V)=-{R}^P(X,Y,V,U).
\end{equation}
Let $^\bot$ denotes the projection on the vector subbundle $(T\calf)^\bot$ orthogonal to $\calf$ in $TM$.
The~Codazzi equation for a foliation (or a submanifold) of $(M,g)$ has the form, e.g.,~\cite{rov-m}:
\begin{equation}\label{E-codazzi-classic}
 (\nabla_{X}\, h)(Y,U) -(\nabla_{Y}\, h)(X,U) = (R(X,Y)U)^\bot,
\end{equation}
where
$R:TM\times TM\rightarrow{\rm End}(TM)$ is the Riemannian curvature tensor of $\nabla$,
\[
 {R}(X,Y) = \nabla_{X}\nabla_{Y} -\nabla_{Y}\nabla_{X} -\nabla_{[X,\, Y]},
\]
and $h:T\calf\times T\calf\to(T\calf)^\bot$
is the second fundamental form of $\calf$ in $(M,g)$ defined by
\[
 h(X,Y)=(\nabla_XY)^\bot.
\]
Hence, $\<A_\xi(X),\,Y\>=\<h(X,Y),\,\xi\>$ for $\xi\in {\mathcal N}\calf$.

\begin{lemma}
The following Codazzi type equation is valid:
\begin{equation}\label{E-codazziAxi}
 (\nabla^\calf_{X}\, A_\xi) Y -(\nabla^\calf_{Y}\, A_\xi) X = -R^P(X,Y)\xi ,\quad X,Y\in T\calf,\ \ \xi\in {\mathcal N}\calf.
\end{equation}
\end{lemma}

\begin{proof}
From \eqref{E-codazzi-classic}, for all vectors $X,Y,U\in T\calf$ we get
\begin{equation}\label{E-codazzi-1}
 \<(\nabla^\calf_{X}\, A_\xi) Y -(\nabla^\calf_{Y}\, A_\xi) X,\ U\> + \<h(X,U), \nabla_Y \xi\> - \<h(Y,U), \nabla_X \xi\>
 = -\<R(X,Y)\xi, U\> .
\end{equation}
 Applying the orthoprojector from $TM$ on the vector subbundle $(T\calf)^\bot$, we find
\begin{equation}\label{E-codazzi-RN}
 \<R(X,Y)\xi, U\>=\<R^P(X,Y)\xi, U\> +
 \<\nabla_{X}((\nabla_{Y}\,\xi)^\bot) -\nabla_{Y}((\nabla_{X}\,\xi)^\bot) -\nabla_{[X,\ Y]^\bot}\,\xi, U\>.
\end{equation}
Using the equalities $[X,\ Y]^\bot=0$ (since $T\calf$ is integrable), \eqref{E-codazzi-RN} and
\begin{eqnarray*}
 \<h(X,U),\, \nabla_Y\,\xi\> = \<\nabla_{Y}\xi,\, (\nabla_{X}\,U)^\bot\> = -\<\nabla_{X}((\nabla_{Y}\,\xi)^\bot),\, U\>,\\
 \<h(Y,U),\, \nabla_X\,\xi\> = \<\nabla_{X}\xi,\, (\nabla_{Y}\,U)^\bot\> = -\<\nabla_{Y}((\nabla_{X}\,\xi)^\bot),\, U\>,
\end{eqnarray*}
in \eqref{E-codazzi-1} completes the proof.
\end{proof}

\section{The $\calf$-divergence of (1,1)-tensors on ${\mathcal D}$}

 An orthonormal frame $\{e_i, e_a\}_{1\le i\le n,\,1\le a\le p}$ on ${\mathcal D}$ built either of single vectors or of (local) vector fields
is said to be \textit{adapted} (to $\calf$), whenever $e_i\in T\calf$ and $e_a\in {\mathcal N}\calf$.
Following \cite{aw2010,r8}, define the $\calf$-divergence of (1,1)-tensors $A^k_\xi$ and
$T_r(A_\xi)$ by
\begin{equation}\label{eq:div1}
 \Div_\calf A^k_\xi = \sum\nolimits_{\,i=1}^n(\nabla^\calf_{e_i}\, A^k_\xi)\,e_i,\quad
 \Div_\calf T_r(A_\xi) = \sum\nolimits_{\,i=1}^n(\nabla^\calf_{e_i}\,T_r(A_\xi))\,e_i.
\end{equation}
Note that
$\Div_\calf T_0(A_\xi) =\Div_\calf\,\id_{T\calf} = \nabla^\calf_{e_i}(\id_{T\calf}\,e_i) -\id_{T\calf}(\nabla^\calf_{e_i}\,e_i)=0$.

\begin{remark}\rm
Given a $(1,1)$-tensor $S$, they usually define the 1-form ${\rm div}_\calf\,S$ by
\begin{equation*}
 (\Div_\calf S)(X) = \sum\nolimits_{\,i=1}^n\<(\nabla_{i}\,S)(X), \,e_i\> ,
\end{equation*}
and the vector field $\nabla^{*\calf}S$ (called the \textit{adjoint of the covariant $\calf$-derivative}) by
\begin{equation*}
 (\nabla^{*\calf} S)(X) =-\sum\nolimits_{\,i=1}^n(\nabla_i\,S)( e_i, X).
\end{equation*}
Using a metric $g$ on a manifold, we identify $(1,0)$-tensors with $(0,1)$-tensors.
For simplicity, in the article we use the definition \eqref{eq:div1} and do not mention the $\nabla^{*\calf} $-notation.
\end{remark}

For any $X\in{\mathcal D}$ and $\xi\in {\mathcal N}_1\calf$, define a linear operator ${\mathcal R}^P_{X, \,\xi}:T\calf\to T\calf$ by
\begin{equation*}
 {\mathcal R}^P_{\,X, \,\xi}: V \to (R^P(V, X)\,\xi)^\top,\quad V\in T\calf.
\end{equation*}
The following result generalizes \cite[Lemma~2.2]{aw2010}, see also \cite[Lemmas~2.4 and 2.5]{r8}.

\begin{proposition}
$($a$)$ The following formula is valid for any $\xi\in {\mathcal N}_1\calf$ and $X\in T\calf$:
\begin{eqnarray}\label{E-divCN-Z}
\nonumber
  \< \Div_\calf{\mathcal A}_\xi, \,X\> \eq \sum\nolimits_{\,k<n}\big[ (A^{k}_\xi X)(f_k)
 +f_k\sum\nolimits_{\,j=1}^k\big(\frac{1}{k-j+1}\,(A^{j-1}_\xi X)(\tau_{k-j+1}(\xi))\\
 && +\tr_{\,\calf}(A^{k-j}_\xi\,{\mathcal R}^P_{\,(-A_\xi)^{j-1} X, \,\xi})
 \big)\big].
\end{eqnarray}
$($b$)$ For~$f_j=(-1)^j\sigma_{r-j}$ and $0<r<n$, \eqref{E-divCN-Z} gives the following:
\begin{equation}\label{E-divNTk-1}
 \<\Div_\calf T_r(A_\xi), \,X\> = \sum\nolimits_{\,j=1}^r
 \tr_{\,\calf}\big(T_{r-j}(A_\xi)\,{\mathcal R}^P_{\,(-A_\xi)^{j-1} X, \,\xi}\big).
\end{equation}
\end{proposition}

\proof (a) We will calculate at a point $x\in M$. One may assume $\nabla_{e_i}\,\xi|_{\,x}\in T_x\calf$ for all $i$.
Decomposing $A_\xi^{k}=A_\xi A_\xi^{k-1}$ for $k\ge1$, we get at a point $x$,
\begin{equation}\label{E-divCk-proof1}
 \Div_\calf A_\xi^{k} = A_\xi \Div_\calf A_\xi^{k-1} +\sum\nolimits_{\,i=1}^n(\nabla^{\calf}_{e_i} A_\xi) A^{k-1}_\xi e_i.
\end{equation}
Since (\ref{E-divCk-proof1}) is tensorial, it is valid for any point of $M$.
Using \eqref{E-codazziAxi}, we find for $X\in T_x\calf$,
\begin{eqnarray*}
 &&\sum\nolimits_{\,i=1}^n\<(\nabla^{\calf}_{e_i}A_\xi)A_\xi^{k-1} e_i, \,X\>
   =\sum\nolimits_{\,i=1}^n\<A^{k-1}_\xi e_i, \,(\nabla^{\calf}_{e_i} A_\xi) X\>\\
 &&=\sum\nolimits_{\,i=1}^n\<A^{k-1}_\xi e_i, (\nabla^{\calf}_{X}A_\xi)\,e_i +R^P(X, e_i)\,\xi
 \> \\
 &&=\tr_{\,\calf}(A^{k-1}_\xi\nabla^{\calf}_{X} A_\xi)
 +\tr_{\,\calf}(A^{k-1}_\xi\,{\mathcal R}^P_{\,X, \,\xi}).
\end{eqnarray*}
For $X\in T_x\calf$ and for $k\ge1$, (\ref{E-divCk-proof1}) gives us
\begin{equation*}
 \<\,\Div_\calf A_\xi^{k}, \,X\> = -\<A_\xi \Div_\calf A_\xi^{k-1}, \,X\> +\tr_{\,\calf}(A_\xi^{k-1}\nabla^{\calf}_{X} A_\xi)
  +\tr_{\,\calf}(A^{k-1}_\xi\,{\mathcal R}^P_{\,X, \,\xi}).
\end{equation*}
 The above and the last identity of Lemma~\ref{L-NTprop}, yield the inductive formula
\begin{eqnarray}\label{E-divCk-proof2}
 \<\Div_\calf A^{k}_\xi, \,X\> = \< \Div_\calf A_\xi^{k-1}, \,-A_\xi X\> +\frac1k\,X(\tau_k(\xi))
 +\tr_{\,\calf}(A^{k-1}_\xi\,{\mathcal R}^P_{\,X, \,\xi}).
\end{eqnarray}
By induction, from (\ref{E-divCk-proof2}) we obtain the following for $k\ge1$:
\begin{equation}\label{E-divCk-1}
  \<\Div_\calf A_\xi^{k}, X\> {=} \sum\nolimits_{\,j=1}^k\big(\frac1{k{-}j{+}1}\,(A^{j-1}_\xi X)(\tau_{k-j+1}(\xi))
  +\tr_{\,\calf}(A^{k-j}_\xi\,{\mathcal R}^P_{\,(-A_\xi)^{j-1}X, \,\xi})
 \big).
\end{equation}
For the (1,1)-tensor ${\mathcal A}_\xi$, see \eqref{E-barCN-1}, we get
\begin{equation*}
 \< \Div_\calf{\mathcal A}_\xi, \,X\>
 =\sum\nolimits_{\,k<n}\big(\<\nabla^{\calf} f_k, A^{k}_\xi X\> +f_k\< \Div_\calf A^{k}_\xi, \,X\>\big).
\end{equation*}
Thus, \eqref{E-divCN-Z} follows from \eqref{E-divCk-1}.

(b) By inductive definition (\ref{E-defNTind}), we obtain the following for $r>0$, compare with \eqref{E-divCk-proof1}:
\begin{equation*}
 \Div_\calf T_r(A_\xi)=\nabla^{\calf}\sigma_r(\xi) - A_\xi \Div_\calf T_{r-1}(A_\xi)
 -\sum\nolimits_{\,i=1}^n(\nabla^{\calf}_{e_i} A_\xi) (T_{r-1}(A_\xi) e_i).
\end{equation*}
Note that the (1,1)-tensor $\nabla^{\calf}_{e_i} A_\xi$ is self-adjoint.
Then, using the third identity of Lemma~\ref{L-NTprop},
and Codazzi type equation \eqref{E-codazziAxi}, we get, compare with \eqref{E-divCk-proof2},
\begin{equation}\label{E-divNT-proof2}
 \<\Div_\calf T_r(A_\xi), \,X\> = \<\Div_\calf T_{r-1}(A_\xi), \,(-A_\xi)X\>
 +\tr_{\,\calf}(T_{r-1}(A_\xi)\,{\mathcal R}^P_{\,X, \,\xi}).
\end{equation}
From \eqref{E-divNT-proof2} by induction we obtain \eqref{E-divNTk-1}.
\hfill$\square$

\begin{example}
\rm
(a) For $k=1$, \eqref{E-divCk-1} reads as
\begin{equation*}
 \<\Div_\calf A_\xi, X\> = X(\tau_{1}(\xi)) +\tr_{\,\calf}{\mathcal R}^P_{X,\xi}.
\end{equation*}
(b) Let the distribution $T\calf$ be $P$-\textit{curvature invariant}, that is
\begin{equation}\label{E-Pcurv-inv}
 R^P(X,Y)V \in T\calf\quad (X,Y,V\in T\calf).
\end{equation}
Then, in view of \eqref{E-R-symm2}, equation \eqref{E-divNTk-1} implies that $\Div_\calf T_r(A_\xi)=0$ for every $r$ and $\xi$.
Note that \eqref{E-Pcurv-inv} is satisfied, if $T\calf$ is
auto-parallel, i.e., $\nabla_XY\in\Gamma(T\calf)$ for all $X,Y\in\Gamma(T\calf)$.

A sufficient condition for \eqref{E-Pcurv-inv} is the following equality:
\begin{equation}\label{E-Pcurv-c}
 R^P(X,Y)Z= c\,(\/\<Y,Z\>X-\<X,Z\>Y\/),\quad X,Y,Z\in{\mathcal D},
\end{equation}
for some constant
$c\in\RR$.
\end{example}

\begin{remark}\rm
 Let $\calf$ be a foliation of a Riemannian manifold $(M,g)$ and ${\mathcal D}=TM$; thus, $R^P=R$.
Using the symmetry $\<R(X,Y)U, \,V\>=\<R(U,V)X, \,Y\>$ of the curvature tensor, we simplify \eqref{E-divNT-proof2} to the following form:
\begin{equation*}
 \Div_\calf T_r(A_\xi) = - A_\xi\Div_\calf T_{r-1}(A_\xi) +\sum\nolimits_{\,i=1}^n (R(\xi, \,T_{r-1}(A_\xi) e_i)e_i)^\top,
\end{equation*}
where $^\top$ denotes the orthogonal projection on the vector subbundle $T\calf$,
see \cite[Proposition~2.9]{r8} and \cite[Lemma~2.2]{aw2010}.
\end{remark}

\section{The $\calf$-divergence of vector fields on ${\mathcal D}$}

For any local vector field $\xi$ in ${\mathcal N}_1\calf$, put
\[
 Z_{\,\xi}=(\nabla_{\xi}\,\xi)^\top.
\]
The $\calf$-{divergence} of a vector~field $X$ on $M$ is defined by
\[
 \Div_\calf X
 =\sum\nolimits_{\,i=1}^n\<\nabla_{e_i}\,X, \,e_i\>,
\]
where $\{e_i\}$ is a local orthonormal frame of $T\calf$.

Let ${\mathcal N}_1\calf\subset{\mathcal N}\calf$ be a subbundle of unit vectors orthogonal to $T\calf$ in~${\mathcal D}$,
$({\mathcal N}_1\calf)_x$ its fiber (a unit sphere) at a point $x\in M$,
${\rm d}\,\omega^\perp_x$ the volume form on $({\mathcal N}_1\calf)_x$ with the induced metric.

Applying Lemma~\ref{L-31-myC} given below,
and (\ref{E-divCN-Z}), we generalize \cite[Proposition~2.8]{r8}.

\begin{proposition}\label{P-Cintfiv01}
The divergence of the vector field $\int_{\,({\mathcal N}_1\calf)_x}{\mathcal A}_\xi Z_{\,\xi}\,{\rm d}\,\omega^\bot_x$
at a point $x\in M$~is
\begin{eqnarray*}
&& \Div_{\calf}\int_{\,({\mathcal N}_1\calf)_x}{\mathcal A}_\xi Z_{\,\xi}\,{\rm d}\,\omega^\bot_x =
 \int_{\,({\mathcal N}_1\calf)_x}\Big(\underline{\<\Div_\calf{\mathcal A}_\xi, \,Z_{\,\xi}\>} +\tr_{\,\calf}({\mathcal A}_\xi {\mathcal R}^P_{\,\xi,\xi}) \\
&& +\sum\nolimits_{\,k<n}\big( f_k\,\tau_{k+2}(\xi) -\frac{f_k}{k+1}\,\xi(\tau_{k+1}(\xi))\big)
+\sum\nolimits_{\,a=1}^p\<\,{\mathcal A}_\xi(\nabla_{e_a} \xi)^\top, \,\nabla_{\xi}\,e_a\>\Big)\,{\rm d}\,\omega^\bot_x,
\end{eqnarray*}
where the underlined term is given in \eqref{E-divCN-Z} with $X=Z_{\,\xi}$. In particular,
the divergence of the vector field $\int_{\,({\mathcal N}_1\calf)_x} T_r(A_\xi)Z_{\,\xi}\,{\rm d}\,\omega^\bot_x$ at a point $x\in M$ is
\begin{eqnarray}\label{E-divFNTh}
\nonumber
 && \Div_{\calf}\int_{\,({\mathcal N}_1\calf)_x} T_r(A_\xi)Z_{\,\xi}\,{\rm d}\,\omega^\bot_x
 =\int_{\,({\mathcal N}_1\calf)_x}\!\Big(\underline{\<\Div_\calf T_r(A_\xi), \,Z_{\,\xi}\>} \\
\nonumber
 && -\,\xi(\sigma_{r+1}(\xi)) -(r+2)\,\sigma_{r+2}(\xi) +\sigma_{1}(\xi)\,\sigma_{r+1}(\xi)+\tr_{\,\calf}(T_r(A_\xi) {\mathcal R}^P_{\,\xi, \,\xi}) \\
 && +\sum\nolimits_{\,a=1}^p\<\,T_r(A_\xi)(\nabla_{e_a} \xi)^\top, \nabla_{\xi}\, e_a\>\Big)\,{\rm d}\,\omega^\bot_x,
\end{eqnarray}
where the underlined term is given by \eqref{E-divNTk-1} with $X=Z_\xi$.
\end{proposition}

\proof
 Assume that $\nabla_X \xi\in T_x\calf$ for any $X\in T_x M$ at a point $x\in M$, and calculate
\begin{eqnarray}\label{E-divChN}
\nonumber
 && \Div_{\calf}\int_{\,({\mathcal N}_1\calf)_x}{\mathcal A}_\xi Z_{\,\xi}\,{\rm d}\,\omega^\bot_x =
 \sum\nolimits_{\,i=1}^n\int_{\,({\mathcal N}_1\calf)_x}\big\<\nabla_{e_i} ({\mathcal A}_\xi Z_{\,\xi}), e_i\big\>\,{\rm d}\,\omega^\bot_x \\
&&=\int_{\,({\mathcal N}_1\calf)_x}\big\< \Div_\calf{\mathcal A}_\xi, \,Z_{\,\xi}\big\>\,{\rm d}\,\omega^\bot_x
 +\sum\nolimits_{\,i=1}^n\int_{\,({\mathcal N}_1\calf)_x}\big\<\nabla^\calf_{e_i}Z_{\,\xi},\, {\mathcal A}_\xi\,e_i\big\>\,{\rm d}\,\omega^\bot_x.
\end{eqnarray}
By Lemma~\ref{L-31-myC} in what follows, we find the integrand
$\sum\nolimits_{\,i=1}^{\,n}\<\nabla^\calf_{e_i}Z_{\,\xi}, {\mathcal A}_\xi\,e_i\>$ in~(\ref{E-divChN}),
\begin{equation}\label{E-termA-nabla-A}
 \sum\nolimits_{\,k<n} f_k\tau_{k+2}(\xi)+\tr_{\,\calf}({\mathcal A}_\xi {\mathcal R}^P_{\,\xi, \,\xi})
 -\tr_{\,\calf}({\mathcal A}_\xi(\nabla^{\calf}_\xi A_\xi))
 +\sum\nolimits_{\,a=1}^p\<{{\mathcal A}_\xi}(\nabla_{e_a} \xi)^\top, \nabla_{\xi}\,e_a\>.
\end{equation}
We transform the term $\tr_{\,\calf}({\mathcal A}_\xi\nabla^\calf_\xi A_\xi)$ in \eqref{E-termA-nabla-A},
using the definition ${\mathcal A}_\xi=\sum_{\,k<n} f_k A_\xi^k$, see \eqref{E-barCN-1}, and the last identity of Lemma~\ref{L-NTprop}, as
\begin{equation*}
 \tr_{\,\calf}({\mathcal A}_\xi\nabla^{\calf}_\xi A_\xi) = \sum\nolimits_{\,k<n}\frac{f_k}{k+1}\,\xi(\tau_{k+1}(\xi)).
\end{equation*}
Finally, we conclude that
\begin{eqnarray*}
 && \sum\nolimits_{\,i=1}^n\<\nabla^\calf_{e_i}Z_{\,\xi}, \,{\mathcal A}_\xi\,e_i\> = \tr_{\,\calf}({\mathcal A}_\xi {\mathcal R}^P_{\,\xi, \,\xi}) \\
 && +\sum\nolimits_{\,k<n} \big(f_k\tau_{k+2}(\xi) -\frac{f_k}{k+1}\,\xi(\tau_{k+1}(\xi))\big)
 +\sum\nolimits_{\,a=1}^p\big\<{{\mathcal A}_\xi}(\nabla_{e_a} \xi)^\top, \,\nabla_{\xi}\,e_a\big\>.
\end{eqnarray*}
Since the result
is tensorial, it is valid for any $x\in M$.
The proof of \eqref{E-divFNTh} is similar.
\hfill$\square$

\begin{remark}\label{R-I00}
\rm The integrals over $({\mathcal N}_1\calf)_x$ when ${p}>1$ and $x\in M$ can be found explicitly.
To show this, denote ${\lambda}=(\lambda_1,\ldots,\lambda_{p})$ and ${y}=(y_1,\ldots, y_{p})$.
Then, see e.g., \cite{r8},
\[
 I_{\lambda}:=\int\nolimits_{\,\|y\|=1}{y}^{\lambda}\,{\rm d}\,\omega_{{p}-1}
 =\frac{2}{\Gamma\big({p}/2+(1/2)\sum_{1\le a\le {p}}\lambda_a\big)}
 \prod\nolimits_{\,1\le a\le {p}}\frac12\,(1+(-1)^{\lambda_a})\,\Gamma\big(\frac{1+\lambda_a}2\big),
\]
where ${y}^{\lambda}=\prod_{\,a\le {p}} y_a^{\lambda_a}$, and $\Gamma$ is the Gamma function.
For example,
\[
 I_{0,\dots, 0} = \frac{2\,\pi^{{p}/2}}{\Gamma({p}/2)}={\rm vol}(S^{{p}-1}(1)),\quad
 I_{2\lambda_1, 0,\dots, 0} = 2\,\pi^{\frac{{p}-1}2}\frac{\Gamma(1/2+\lambda_1)}{\Gamma({p}/2+\lambda_1)}.
\]
\end{remark}

The following lemma generalizes \cite[Lemma~3.1]{aw2010}, see also \cite[Lemma~2.7]{r8}.

\begin{lemma}\label{L-31-myC}
Let $\{e_i, e_a\}$ be an adapted orthonormal frame of ${\mathcal D}$
 such that

\smallskip
$\bullet$ $\nabla^\calf_X\,e_i=0\ (1\le i\le n)$ and $\<\nabla_X e_a,\,e_b\>=0\ (1\le a,b\le p)$ for any vector $X\in T_xM$;

\smallskip

$\bullet$ $\nabla^P_\mu\,e_i=0\ (1\le i\le n)$ for any vector $\mu\in \widetilde{\mathcal D}_x$ at a point $x\in M$.

\smallskip
\noindent
Then for any unit vector $\xi\in ({\mathcal N}_1\calf)_x$ we have
\begin{equation*}
 \<\nabla^{\calf}_{e_i} Z_{\,\xi}, e_j\> = \<A^2_\xi e_{i}, e_{j}\> {+} \<R^P(e_i,\xi)\xi, e_j\> {-}\<(\nabla^{\calf}_\xi \,A_\xi)e_i,e_j\>
 +\!\sum\nolimits_{\,a=1}^p\<{\nabla_{\xi}\,e_a}, e_i\>\<\nabla_{e_a}\xi, e_j\>.
\end{equation*}
\end{lemma}

\proof
Taking covariant derivative of $\<Z_{\,\xi},\, e_j\>=-\<\xi,\, \nabla_{\xi}\,e_j\>$ with respect to $e_i$, we find
\begin{equation}\label{E-aw-08C}
 -\<Z_{\,\xi}, \nabla_{e_i}e_j\>= \<\nabla_{e_i} Z_{\,\xi}, e_j\> +\<\nabla_{e_i}\xi, P\nabla_{\xi}\,e_j\> +\<\xi, \nabla_{e_i}P\nabla_{\xi}\,e_j\>.
\end{equation}
For a foliation $\calf$, we obtain
\[
 \<A^2_\xi e_{i}, e_{j}\> = \nabla_{\xi}\<\xi, \nabla_{e_i}\,e_j\>=\<Z_{\,\xi}, \nabla_{e_i}e_j\> +\<\xi,\nabla_{\xi}P\nabla_{e_i}\,e_j\>.
\]
Therefore, 
we calculate at the point $x\in M$:
\begin{eqnarray}\label{E-aw-09C}
\nonumber
 && \<A^2_\xi e_{i}, e_{j}\> +\<R^P(e_i,\xi)\,\xi, e_j\>-\<(\nabla^{\calf}_\xi A_\xi)e_i, e_j\> \\
\nonumber
 && =\<A^2_\xi e_{i}, e_{j}\> -\<R^P(e_i,\xi)\,e_j, \xi\>+\xi\<\nabla_{e_i} \xi, \,e_j\>\\
 && =\<A^2_\xi e_{i}, e_{j}\> -\<Z_{\,\xi}, \nabla_{e_i}e_j\>-\<\nabla_{e_i}P\nabla_{\xi}\, e_j, \xi\> +\<\nabla_{[e_i,\xi]}\,e_j, \xi\>.
\end{eqnarray}
Using (\ref{E-aw-08C}) and the following equalities at $x\in M$:
\begin{eqnarray*}
 && P\nabla_{e_i}\xi=\sum\nolimits_{\,j=1}^n\<\nabla_{e_i}\xi, e_j\>e_j,\quad
 P\nabla_{\xi}\,e_i=\sum\nolimits_{\,a=1}^p\<\nabla_{\xi}\, e_i, e_a\>\,e_a,\\
 &&
 \<A^2_\xi\, e_{i}, e_{j}\>=-\sum\nolimits_{\,k=1}^n\<\nabla_{e_i}\xi, e_k\>\<\nabla_{e_k}e_j, \xi\>,
\end{eqnarray*}
we simplify the last line in (\ref{E-aw-09C}) as
\begin{equation*}
 \<\nabla_{e_i}Z_{\,\xi}, e_j\>-\sum\nolimits_{\,k=1}^n\<\nabla_{\xi}\, e_k, e_i\>\<\nabla_{e_k}\, \xi,e_j\>.
\end{equation*}
From the above, the claim follows.
\hfill$\square$

\section{Main results}
\label{sec:main}

Here, we prove integral formulas for a foliated sub-Riemannian mani\-fold $(M,{\mathcal D},\calf,g)$.

The idea is to find the divergence of a suitable vector field
and
apply the Divergence Theorem and Remark~\ref{lem:Sas3} in what follows to $(M,g)$ or to a compact leaf with the induced~metric.
We will integrate over the normal sphere bundle ${\mathcal N}_1\calf\subset{\mathcal N}\calf$ of $\calf$ with the induced metric.

\begin{remark}\label{lem:Sas3}\rm
 For any function $f:{\mathcal N}_1\calf\to{\mathbb R}$, we have the Fubini formula, see \cite[Note 7.1.1.1]{ber},
 \begin{equation*}
 \int_{{\mathcal N}_1\calf}f(\xi)\,{\rm d}\,\omega^\bot = \int_{M}\big(\int_{\,({\mathcal N}_1\calf)_x}f(\xi)\,{\rm d}\,\omega^\bot_x\big)\,{\rm d}\vol_{g},
 \end{equation*}
where
${\rm d}\,\omega^\perp$ is the volume form on ${\mathcal N}_1\calf$
and ${\rm d}\vol_g$ is the volume form of $(M,g)$.
\end{remark}

The next theorem with integral formulas along a compact leaf of a foliation
follows from Proposition~\ref{P-Cintfiv01} and generalizes \cite[Theorem~3.1]{r8}.

\begin{theorem}\label{T-mainLW}
Let $(M,{\mathcal D},\calf,g)$ be a foliated sub-Riemannian manifold with ${\mathcal D}=T\calf\oplus {\mathcal N}\calf$.
Then for any compact leaf $L$ of $\calf$ we have
\begin{eqnarray}\label{E-intCh-B}
\nonumber
 &&\int_{{\mathcal N}_1\calf|_{L}}\big(\<\underline{\Div_\calf {\mathcal A}_\xi, \,Z_{\,\xi}}\>
 +\sum\nolimits_{\,k<n}\big(f_k\,\tau_{k+2}({\xi}) -\frac{f_k}{k+1}\,{\xi}(\tau_{k+1}({\xi}))\big) \\
 &&
 +\tr_{\,\calf}({\mathcal A}_{\,\xi}\,{\mathcal R}^P_{{\xi},\xi})
 +\sum\nolimits_{\,a=1}^p\<{\mathcal A}_{\,\xi}(\nabla_{e_a} \xi)^\top, \,\nabla_{\xi}\,e_a\>\big)\,{\rm d}\,\omega^\perp_L =0,
\end{eqnarray}
where ${\rm d}\,\omega^\perp_L$ is the volume form on ${\mathcal N}_1\calf|_{L}$,
and the underlined term is given by \eqref{E-divCN-Z} with $X=Z_\xi$.
For~$f_j=(-1)^j\sigma_{r-j}$, \eqref{E-intCh-B} gives
\begin{eqnarray}\label{E-intCh-C}
 &&\int_{{\mathcal N}_1\calf|_{L}}
 \nonumber
 \big(\<\underline{\Div_\calf T_r(A_\xi), \,Z_{\,\xi}}\> -\xi(\sigma_{r+1}(\xi)) -(r+2)\,\sigma_{r+2}(\xi)
 +\sigma_{1}(\xi)\,\sigma_{r+1}(\xi)\\
 && \ +\,\tr_{\,\calf}(T_r(A_\xi) {\mathcal R}^P_{\,\xi, \,\xi})
 +\sum\nolimits_{\,a=1}^p\<T_r(A_\xi)(\nabla_{e_a} \xi)^\top, \,\nabla_{\xi} \,e_a\>\big)\,{\rm d}\,\omega^\perp_L =0,
\end{eqnarray}
where the underlined term is given by \eqref{E-divNTk-1} with $X=Z_\xi$.
\end{theorem}

For any vector field $X$ in the distribution $T\calf$, we have
\begin{equation}\label{E-HD}
 \Div X = \Div_\calf X -\<X, \,{H}^\bot\> -\<X, \,\widetilde{H}\>,
\end{equation}
where
${H}^\bot$
is the mean curvature vector of ${\mathcal N}\calf$,
 and
$\widetilde{H}$ is the mean curvature vector field of $\widetilde{\mathcal D}$.
Recall that $\widetilde{\mathcal D}$ is a \textit{harmonic} distribution if $\widetilde{H}=0$.
There are topological restrictions for the existence of a Riemannian metric on closed manifold,
for which a given distribution is harmonic, see~\cite{su2}.
 Using Proposition~\ref{P-Cintfiv01}, we get the following
 theorem with integral formulas along a closed foliated manifold, which generalizes \cite[Theorem~3.3]{r8}.

\begin{theorem}\label{T-mainLW-2}
Let $(M,{\mathcal D},\calf,g)$ be a foliated closed sub-Riemannian manifold with ${\mathcal D}=T\calf\oplus {\mathcal N}\calf$
and a harmonic distribution $\widetilde{\mathcal D}$.
Then the following integral formula is valid:
\begin{eqnarray}\label{E-intCh}
\nonumber
 && \int_{{\mathcal N}_1\calf} \Big(\<\underline{\Div_\calf{\mathcal{A}}_\xi, \,Z_\xi}\>
 +\sum\nolimits_{k<n}\big(f_k\,\tau_{k+2}(\xi) +\frac{\tau_{k+1}(\xi)}{k+1}\,(\xi(f_k)-f_k\tau_1(\xi))\big)\\
 &&\hskip4mm +\tr_{\,\calf}({\mathcal{A}}_\xi {\cal R}^P_{\xi,\xi})-\<{\mathcal{A}}_\xi Z_\xi, H^\perp\>
 +\sum\nolimits_{\,a=1}^p\<{{\mathcal{A}}_\xi}(\nabla_{e_a} \xi)^\top, \,\nabla_{\xi}\,e_a\>\Big) {\rm d}\,\omega^\perp = 0,
\end{eqnarray}
where
the underlined term is given by \eqref{E-divCN-Z} with $X=Z_\xi$.
For $f_j=(-1)^j\sigma_{r-j}$, \eqref{E-intCh} gives
\begin{eqnarray}\label{E-intNTNT}
\nonumber
 && \int_{{\mathcal N}_1\calf}\Big(\< \underline{\Div_\calf T_r(A_{\xi}), \,Z_{\,\xi}}\> -(r+2)\,\sigma_{r+2}({\xi}) -\<T_r(A_{\xi}) Z_{\,\xi}, {H}^\bot\>\\
 && +\tr_{\,\calf}(T_r(A_{\xi})\,{\mathcal R}^P_{\,{\xi}, \,\xi})+\sum\nolimits_{\,a=1}^p\<T_r(A_{\xi})(\nabla_{e_a}{\xi})^\top, \,\nabla_{{\xi}}\,e_a\>\Big)
 \,{\rm d}\,\omega^\perp = 0,
\end{eqnarray}
where the underlined term is given by \eqref{E-divNTk-1} with $X=Z_\xi$.
\end{theorem}

\begin{proof}
 Using $\sum_{\,1\le a\le p} (\nabla_{e_a} e_a)^\top=P(H^\perp)$ and \eqref{E-HD} with $X(x)=\int_{\,({\mathcal N}_1\calf)_x}{\mathcal A}_{\xi} Z_\xi\,{\rm d}\,\omega_x^\perp$ and our assumption $\widetilde H=0$, we get
\begin{eqnarray}\label{E-divdiv}
\nonumber
  \Div\int_{({\mathcal N}_1\calf)_x}\!{{\mathcal A}_{\xi}} Z_\xi\,{\rm d}\,\omega_x^\perp
  \eq\Div_\calf\big(\int_{({\mathcal N}_1\calf)_x}\!{{\mathcal A}_{\xi}} Z_\xi\,{\rm d}\,\omega_x^\perp\big)
  +\sum\limits_{a=1}^p\int_{({\mathcal N}_1\calf)_x}\!\<\nabla_{e_a}({\mathcal A}_{\xi} Z_\xi), \,e_a\>\,{\rm d}\,\omega_x^\perp\\
  \eq\Div_\calf\big(\int_{\,({\mathcal N}_1\calf)_x}{\mathcal A}_{\xi} Z_\xi\,{\rm d}\,\omega_x^\perp\big)
  -\int_{\,({\mathcal N}_1\calf)_x}\<{\mathcal A}_{\xi} Z_\xi, \,H^\perp\>\,{\rm d}\,\omega_x^\perp.
\end{eqnarray}
Assuming $\nabla_X {\xi}\perp{\mathcal N}\calf$ for any $X\in T_x M$ and $\xi\in({\mathcal N}_1\calf)_x$ at some point $x\in M$, we get
\[
 \Div {\xi}=\Div_\calf {\xi}=-\tau_1({\xi}).
\]
Using this equality,
we also find
\begin{eqnarray*}
 &&\hskip-6mm\Div\big(\frac{f_k}{k+1}\,\tau_{k+1}({\xi})\,{\xi}\big)
 =\frac{f_k}{k+1}\,\Div\big(\tau_{k+1}({\xi})\,{\xi}\big)+\frac{1}{k+1}\,\tau_{k+1}({\xi})\,{\xi}(f_k)\\
 &&=\frac{f_k}{k+1}\big(\tau_{k+1}({\xi})\Div {\xi} +{\xi}(\tau_{k+1}({\xi}))\big) +\frac{1}{k+1}\,\tau_{k+1}({\xi})\,{\xi}(f_k)\\
 &&=\frac{f_k}{k+1}\,{\xi}(\tau_{k+1}({\xi}))+\frac{1}{k+1}\,\tau_{k+1}({\xi})({\xi}(f_k)-f_k\tau_1({\xi})).
\end{eqnarray*}
Then at the point $x\in M$ we obtain the equality for tensors
\begin{eqnarray*}
 &&\Div\int_{\,({\mathcal N}_1\calf)_x}\big({{\mathcal A}_{\xi}} Z_\xi
 +\sum\nolimits_{\,k<n}\frac{f_k\tau_{k+1}({\xi})}{k+1}\,\xi\big)\,{\rm d}\,\omega_x^\perp
 =\int_{\,({\mathcal N}_1\calf)_x}\Big(\underline{\<\Div_\calf{\mathcal A}_{\xi}, Z_\xi\>} \\
 &&\hskip1mm
 +\sum\nolimits_{\,k<n}\big(f_k\tau_{k+2}({\xi})+\frac{1}{k+1}\,\tau_{k+1}({\xi})\,({\xi}(f_k)-f_k\tau_1({\xi}))\big)  \\
 &&\hskip1mm
 +\,\tr_{\,\calf}({\mathcal A}_{\xi}\,{\mathcal R}^P_{\,\xi, \,\xi})-\<{{\mathcal A}_{\xi}} Z_\xi, H^\perp\>
 +\sum\nolimits_{\,a=1}^p\<{{\mathcal A}_{\xi}}(\nabla_{e_a} {\xi})^\top, \,\nabla_{\xi}\,e_a\>\Big)\,{\rm d}\,\omega_x^\perp,
\end{eqnarray*}
see (\ref{E-divCN-Z}) for the underlined term.
Applying the Divergence Theorem yields~(\ref{E-intCh}).

\vskip1mm
For the Newton transformations of $A_\xi$, similarly to (\ref{E-divdiv}), we have
\begin{equation*}
  \Div\int_{\,({\mathcal N}_1\calf)_x} T_r(A_\xi) Z_\xi\,{\rm d}\,\omega_x^\perp
  =\Div_\calf\int_{\,({\mathcal N}_1\calf)_x} (T_r(A_\xi) Z_\xi -\< T_r(A_\xi) Z_\xi, H^\perp\>)\,{\rm d}\,\omega_x^\perp.
\end{equation*}
Observe that
\begin{equation*}
\Div(\sigma_{r+1}({\xi})\,{\xi})=-\sigma_{1}({\xi})\,\sigma_{r+1}({\xi})+{\xi}(\sigma_{r+1}({\xi}))\quad \mbox{\rm for all } \ {\xi}\in{\mathcal N}_1\calf.
\end{equation*}
 Then at a point $x\in M$ we obtain
\begin{eqnarray*}
\nonumber
 &&\hskip-5mm\Div\!\int_{\,({\mathcal N}_1\calf)_x}\big(T_r(A_\xi) Z_\xi +\sigma_{r+1}({\xi})\,{\xi}\big) {\rm d}\,\omega_x^\perp
 =\hskip-.5mm\int_{\,({\mathcal N}_1\calf)_x}\!\Big(\<\underline{\Div_\calf T_r(A_\xi), Z_\xi}\>-(r+2)\,\sigma_{r+2}({\xi})\\
 &&\hskip2mm -\,\<T_r(A_\xi) Z_\xi, H^\perp\> +\tr_{\,\calf}(T_r(A_\xi)\,{\mathcal R}^P_{\,\xi, \,\xi})
 +\sum\nolimits_{\,a=1}^p\<T_r(A_\xi)(\nabla_{e_a} {\xi})^\top, \nabla_{\xi}\,e_a\>\Big)\,{\rm d}\,\omega_x^\perp,
\end{eqnarray*}
see \eqref{E-divFNTh} for the underlined term.
Using the Divergence Theorem yields~(\ref{E-intNTNT}).
\end{proof}

\section{Applications of main results}

Here, we apply results of Section~\ref{sec:main} to
foliations
with
small $k$ or $r$, with restrictions on the curvature and extrinsic geometry and to codimension-one foliations.

Let $(M,{\mathcal D},\calf,g)$ be a foliated sub-Riemannian manifold with ${\mathcal D}=T\calf\oplus {\mathcal N}\calf$, then

(i) ${\mathcal N}\calf$ will be called $P$-\textit{auto-parallel}, if $(\nabla_XY)^\top = 0$ for all $X,Y\in{\mathcal N}\calf$,

(ii) $\calf$ will be called $P$-\textit{harmonic}, if
$\sigma_1(\xi)=0$ for all $\xi\in {\mathcal N}_1\calf$,

(iii) $\calf$ will be called $P$-\textit{totally umbilical}, if
\begin{equation}\label{E-P-umb}
 A_\xi=(\sigma_1(\xi)/n)\id_{\,T\calf},\quad \xi\in {\mathcal N}_1\calf.
\end{equation}
Obviously, $P$-{auto-parallel}, $P$-{harmonic} and $P$-{totally umbilical} distributions are
{auto-paral\-lel}, {harmonic} and {totally umbilical}, respectively, but the opposite is not true.

\smallskip

The~\textit{mixed scalar $P$-curvature} ${\rm S}^P_{\,\rm mix}$
-- the simplest curvature invariant of a foliated sub-Riemannian manifold --
is defined as an ave\-raged sum of sectional $P$-curvatures of mixed planes (i.e., 2-planes that non-trivially intersect with each of the distributions
$T\calf$ and ${\mathcal N}\calf$);
for an adapted orthonormal frame, we have
\begin{equation*}
 {\rm S}^P_{\,\rm mix}
 =\sum\nolimits_{\,a=1}^p\tr_{\,\calf}{\cal R}^P_{e_a,e_a}
 = \sum\nolimits_{\,a=1}^p\sum\nolimits_{\,i=1}^n \<R^P(e_i, e_a)e_a,\, e_i\>.
\end{equation*}
Let $h^\bot,T^\bot: {\mathcal N}\calf\times {\mathcal N}\calf\to ({\mathcal N}\calf)^\bot$ and $H^\bot=\tr_g h^\bot$
be the second fundamental form, the integrability tensor and the mean curvature vector field of
${\mathcal N}\calf$, and similarly for $T\calf$; in~our case of a foliation, we have $T=0$.
The following corollary generalizes \eqref{E-PW-int}.

\begin{corollary}\label{C-main00}
From \eqref{E-intCh} for ${\mathcal A}_\xi=\id_{\,T\calf}$ we get the following integral formula over $M$:
\begin{equation}\label{E-PW2-int}
 \int_M\big({\rm S\,}^P_{\rm mix} +\|P\circ h\|^2+\|P\circ h^\bot\|^2-\|P H\|^2-\|P H^\bot\|^2
 -\|P\circ T^\bot\|^2\big)\,{\rm d}\vol_g=0.
\end{equation}
\end{corollary}

\begin{proof}
From (\ref{E-intCh}) with ${\mathcal A}_\xi=\id_{\,T\calf}$, we get
\begin{equation}\label{E-intNTNT-k0}
 \int_{{\mathcal N}_1\calf} \Big(
 \tau^2_{1}(\xi)-\tau_{2}(\xi)
 -\tr_{\,\calf}{\cal R}^P_{\xi,\xi}+\<Z_\xi, H^\perp\> -\sum\nolimits_{\,a=1}^p\<(\nabla_{e_a} \xi)^\top, \,\nabla_{\xi}\,e_a\>\Big) {\rm d}\,\omega^\perp = 0.
\end{equation}
Note that $\tau^2_{1}(\xi)-\tau_{2}(\xi) = 2\,\sigma_2(\xi)$.
The same \eqref{E-intNTNT-k0} we get from (\ref{E-intNTNT}) with $r=0$.

Let $\xi=\sum_{\,a=1}^p y_a e_a$, where $y_a\in\RR$, be any unit vector field in ${\mathcal N}\calf$.
For a 2-homogeneous on $\xi$ function
$f(\xi,\xi)=\sum\nolimits_{\,a,b=1}^pf(e_a,e_b)\,y_a y_b$,
as is integrand of (\ref{E-intNTNT-k0}), we have
\[
 \int_{({\mathcal N}_1\calf)_x} f(\xi,\xi)\,{\rm d}\,\omega_x^\perp =\tilde I_2\sum\nolimits_{\,a=1}^p f(e_a,e_a), \
 {\rm where}\ \tilde I_2 =\int_{({\mathcal N}_1\calf)_x} y_a^2\,{\rm d}\,\omega_x^\perp =2\pi^{\frac{{p}-1}2}\frac{\Gamma(3/2)}{\Gamma({p}/2+1)},
\]
see Remark~\ref{R-I00}. Applying this to the terms of (\ref{E-intNTNT-k0}), we find
\begin{eqnarray*}
 && \int_{({\mathcal N}_1\calf)_x}\big(\tau_2(\xi)-\tau_1^2(\xi)\big)\,{\rm d}\,\omega_x^\perp=\tilde I_2(\|P\circ h\|^2-\|P H\|^2),
 \\
&& \int_{({\mathcal N}_1\calf)_x}\<Z_\xi, H^\perp\>\,{\rm d}\,\omega_x^\perp=\tilde I_2\|P H^\perp\|^2,\\
&& \int_{({\mathcal N}_1\calf)_x}\tr_{\,\calf}{\cal R}^P_{\xi,\xi}\,{\rm d}\,\omega_x^\perp=\tilde I_2\sum\nolimits_{\,a=1}^p\tr_{\,\calf}{\cal R}^P_{e_a,e_a}
 =\tilde I_2\,{\rm S\,}^P_{\rm mix},\\
&& \int_{({\mathcal N}_1\calf)_x}\sum\nolimits_{\,a=1}^p\<(\nabla_{e_a} \xi)^\top,\,\nabla_{\xi}\,e_a\>\,{\rm d}\,\omega_x^\perp
 =\tilde I_2\,(\|P\circ h^\perp\|^2-\|P\circ T^\perp\|^2).
\end{eqnarray*}
Thus, \eqref{E-intNTNT-k0} reduces to follows \eqref{E-PW2-int}.
\end{proof}

\begin{example}\rm
We shall look at first members of series (\ref{E-intCh}).
For ${\mathcal A}_\xi=A^2_{\,\xi}$, we get
\begin{eqnarray}\label{E-intCh2}
\nonumber
 && \int_{{\mathcal N}_1\calf} \Big(\<\underline{\Div_\calf A^2_{\,\xi}, \,Z_\xi}\> +\tau_{4}(\xi) -\frac13\,\tau_{3}(\xi)\,\tau_1(\xi)
 +\tr_{\,\calf}(A^2_{\,\xi} {\cal R}^P_{\xi,\xi})\\
 && -\,\<A^2_{\,\xi} Z_\xi, H^\perp\>
 +\sum\nolimits_{\,a=1}^p\<A^2_{\,\xi}(\nabla_{e_a} \xi)^\top, \,\nabla_{\xi}\,e_a\>\Big) {\rm d}\,\omega^\perp = 0.
\end{eqnarray}
Next, we look at first members of
(\ref{E-intNTNT}).
 For $r=2$, (\ref{E-intNTNT}) gives
\begin{eqnarray}\label{E-intNTNT-k2}
\nonumber
 &&\int_{{\mathcal N}_1\calf}\Big(4\,\sigma_{4}({\xi})+\<T_2(A_{\xi})Z_{\,\xi},{H}^\bot\>
 -\tr_{\,\calf}\big(T_2(A_{\xi}){\mathcal R}^P_{\,{\xi}, \,\xi} + T_1(A_{\xi}){\mathcal R}^P_{\,Z_{\,\xi}, \,\xi}
 -{\mathcal R}^P_{\,A_{\xi}Z_{\,\xi},\,\xi}\big) \\
 &&\quad -\sum\nolimits_{\,a=1}^p\<T_2(A_{\xi})(\nabla_{e_a} {\xi})^\top,\,\nabla_{\xi}\,e_a\>\Big)\,{\rm d}\,\omega^\perp = 0.
\end{eqnarray}
If the distribution ${\mathcal N}\calf$ is a $P$-auto-parallel,
 then (\ref{E-intCh2}) and (\ref{E-intNTNT-k2})
 shorten to the formulas
\begin{eqnarray}\label{E-intCh2-geodesic}
 \int_{{\mathcal N}_1\calf} \Big(\tau_{4}(\xi) -\frac13\,\tau_{3}(\xi)\,\tau_1(\xi) +\tr_{\,\calf}(A^2_{\,\xi}\,{\cal R}^P_{\xi,\xi})\Big)
 \,{\rm d}\,\omega^\perp = 0,\\
\label{E-intNTtotG2}
 \int_{{\mathcal N}_1\calf}\Big(\sigma_{4}({\xi})-\frac14\,\tr_{\,\calf}(T_2(A_{\xi})\,{\mathcal R}^P_{\,\xi, \,\xi})\Big)\,{\rm d}\,\omega^\bot=0.
\end{eqnarray}
Similarly to \eqref{E-intNTNT-k0}, one can transform \eqref{E-intCh2}--\eqref{E-intNTtotG2} to the integrals over $M$.
\end{example}

\begin{corollary}\label{C-main01}
Let ${\mathcal N}\calf$ be $P$-auto-parallel, $\widetilde H=0$ and ${\rm S\,}^P_{\rm mix}\ge0$. Then $\calf$ has no compact $P$-harmonic leaves.
\end{corollary}

\begin{proof}
We simplify (\ref{E-intCh-B}) for ${\mathcal A}_\xi=\id_{\,T\calf}$ (the same gives \eqref{E-intCh-C} for $r=0$) as
\begin{eqnarray*}
 &&\int_{{\mathcal N}_1\calf|_{L}}\Big(\tau_{2}({\xi}) -{\xi}(\tau_{1}({\xi}))
 +\tr_{\,\calf}{\mathcal R}^P_{{\xi},\xi} +\sum\nolimits_{\,a=1}^p\<(\nabla_{e_a} \xi)^\top, \,\nabla_{\xi}\,e_a\>\Big)\,{\rm d}\,\omega^\perp_L =0.
\end{eqnarray*}
Thus, the claim follows.
\end{proof}

\smallskip

Now, we apply results of Section~\ref{sec:main} to foliated sub-Riemannian manifolds with restrictions on the
extrinsic geometry of $\calf$.
By \eqref{E-PW2-int}, we easily obtain the following.

\begin{corollary}\label{C-main02}
A closed sub-Riemannian manifold $(M,{\mathcal D},g)$ with
 a harmonic distribution $\widetilde{\mathcal D}$ does not admit

\noindent\ \
{\rm (i)} $P$-harmonic foliations $\calf$
with $P$-{auto-parallel} ${\mathcal N}\calf$ and condition ${\rm S\,}^P_{\rm mix}>0$.

\noindent\ \
{\rm (ii)} $P$-totally umbilical foliations $\calf$
with $P$-{auto-parallel} ${\mathcal N}\calf$ and condition ${\rm S\,}^P_{\rm mix}<0$.
\end{corollary}


From Theorem~\ref{T-mainLW-2} we obtain the following.

\begin{corollary}
Let $(M,{\mathcal D},g)$ be a closed sub-Riemannian manifold with ${\mathcal D}=T\calf\oplus {\mathcal N}\calf$,
a harmonic distribution $\widetilde{\mathcal D}$ and a $P$-{auto-parallel} distribution ${\mathcal N}\calf$.
Then the following integral formulas are valid:
\begin{eqnarray}\label{E-intChR-int}
 &&\int_{\,{\mathcal N}_1\calf}\Big(\tau_{k+2}({\xi})-\frac{1}{k+1}\,\tau_{k+1}({\xi})\,\tau_1({\xi})+\tr_{\,\calf}(A^k_{\xi}\,{\mathcal R}^P_{\,\xi, \,\xi})\Big){\rm d}\,\omega^\bot=0,\\
\label{E-intChR-int2}
 &&\int_{\,{\mathcal N}_1\calf}\Big( (r+2)\,\sigma_{r+2}({\xi})-\tr_{\,\calf}(T_r(A_{\xi})\,{\mathcal R}^P_{\,\xi, \,\xi})\Big) {\rm d}\,\omega^\bot=0.
\end{eqnarray}
\end{corollary}

The \textit{total $k$-th mean curvature\/} of $\calf$ are defined by
\[
 \sigma_{k}(\calf)=\int_{{\mathcal N}_1\calf}\sigma_{k}(\xi)\,{\rm d}\,\omega^\bot,\quad
 \tau_{k}(\calf)=\int_{{\mathcal N}_1\calf}\tau_{k}(\xi)\,{\rm d}\,\omega^\bot.
\]
Note that $\sigma_{2s+1}(\calf)= \tau_{2s+1}(\calf)=0$.
 We have, see Remark~\ref{R-I00},
\[
 \sigma_{0}(\calf) = \tau_{0}(\calf) = n\,\frac{2\pi^{{p}/2}}{\Gamma({p}/2)}\,{\rm Vol}(M,g).
\]



The following corollary of Theorem~\ref{T-mainLW} generalizes \cite[Theorem~1.1]{blr} and \cite[Corollary~4.3]{r8}
on foliated Riemannian manifolds, see also \cite{aw2,rw2} (and \cite[Corollary~1]{r-subR1}, \cite[Section~4.1]{aw2010} for codimension-one foliated manifolds).

\begin{corollary}\label{C-P1}
Let $(M,{\mathcal D},g)$ be a closed sub-Riemannian manifold with ${\mathcal D}=T\calf\oplus {\mathcal N}\calf$.
Suppose that~${\mathcal N}\calf$ is $P$-{auto-parallel}, $\widetilde{\mathcal D}$ is harmonic,
and condition \eqref{E-Pcurv-c} is satisfied.
Then $\sigma_{r}(\calf)$
depends on $r,n,p,c$ and the volume of $(M,g)$ only, i.e., the following integral formula is valid:
 \begin{equation}\label{E-BNtype-1}
 \sigma_{r}(\calf)=
 \bigg\{\begin{array}{cc}
   \frac{2\,\pi^{{p}/2}}{\Gamma({p}/2)}\,\big(\begin{smallmatrix}
     n/2\\
     r/2
   \end{smallmatrix}\big)\,c^{r/2}\,{\rm Vol}(M,g), & n \ {\rm and}\ r \ {\rm even}, \\
   0, &  \ {\rm either}\ n \ {\rm or}\ r \ {\rm odd}.
 \end{array}
\end{equation}
Moreover, if $\calf$ is $P$-harmonic, then
 \begin{equation}\label{E-BNtype-2}
 \tau_{k}(\calf)=
 \bigg\{\begin{array}{cc}
   \frac{2\,\pi^{{p}/2}}{\Gamma({p}/2)}
   \,(-c)^{k/2}\,{\rm Vol}(M,g), & n \ {\rm and}\ k \ {\rm even}, \\
   0, &  \ {\rm either}\ n \ {\rm or}\ k \ {\rm odd}.
 \end{array}
\end{equation}
\end{corollary}

\begin{proof}
\rm
By \eqref{E-Pcurv-c}, ${\mathcal R}^P_{\xi,\xi} = c\,\id_\calf$, hence by Lemma~\ref{L-NTprop},
\[
 \tr_{\,\calf}(A^k_{\xi}\,{\mathcal R}^P_{\,\xi, \,\xi})= c\,\tau_k(\xi),\quad
 \tr_{\,\calf}(T_r(A_{\xi})\,{\mathcal R}^P_{\,\xi, \,\xi})= c\,(n-r)\,\sigma_r(\xi).
\]
By (\ref{E-intChR-int}) and conditions, we get
 $\tau_{k+2}(\calf) = -c\,\tau_{k}(\calf)$.
By (\ref{E-intChR-int2}), we get
\[
 \sigma_{r+2}(\calf)=\frac{c(n-r)}{r+2}\,\sigma_r(\calf).
\]
Then (by induction), from the above we obtain \eqref{E-BNtype-1} and \eqref{E-BNtype-2}.
\end{proof}

We get similar formulas when $(M,{\mathcal D},g)$ is a $P$-\textit{Einstein manifold}, i.e., has the property
\begin{equation}\label{E-P-Einst}
 \tr_{\,\calf}{\mathcal R}^P_{X,\xi}=C\,\<X,\,\xi\>, \quad X\in{\mathcal D},\ \ \xi\in {\mathcal N}_1\calf
\end{equation}
for some $C\in\RR$, and a foliation $\calf$ is $P$-totally umbilical, see \eqref{E-P-umb}.
Note that for ${\mathcal D}=TM$, condition \eqref{E-P-Einst} is satisfied for Einstein manifolds $(M,g)$.

The following corollary of \eqref{E-intNTNT} generalizes result in \cite[Example~4.4]{r8}
(see also \cite[Corollary~2]{r-subR1} and \cite[Section~4.2]{aw2010} for codimension-one foliated manifolds).

\begin{corollary}
Let $(M,{\mathcal D},g)$ be a closed sub-Riemannian manifold with ${\mathcal D}=T\calf\oplus {\mathcal N}\calf$.
Suppose that~${\mathcal N}\calf$ is $P$-{auto-parallel}, $\widetilde{\mathcal D}$ is harmonic,
and conditions \eqref{E-P-umb} and \eqref{E-P-Einst} are satisfied.
Then $\sigma_{r}(\calf)$
depends on $r,n,p,C$ and the volume of $(M,g)$ only, i.e., the following integral formula is valid:
\begin{equation}\label{E-cor2}
  \sigma_{r}(\calf) =
  \bigg\{\begin{array}{cc}
   (C/n)^{r/2}\big(\begin{smallmatrix} n/2 \\ r/2 \end{smallmatrix}\big) {\rm Vol}(M,g), & n,r\ {\rm even}, \\
   0 , & r\ {\rm odd}.
  \end{array}
\end{equation}
\end{corollary}

\begin{proof}
In this case, $T_r(A_\xi)$ has the form
\[
 T_r(A_\xi) = \frac{n-r}{n}\,\sigma_r(\xi)\id_{\,T\calf}.
\]
For a $P$-Einstein manifold, by \eqref{E-P-Einst} we get
\[
 \tr_{\,\calf}{\mathcal R}^P_{Z_\xi,\xi} =0,\quad
 \tr_{\,\calf}{\mathcal R}^P_{\xi,\xi} = C,
\]
where $C\ge0$ by Corollary~\ref{C-main02}.
Thus, \eqref{E-intNTNT} becomes
\begin{equation*}
 \sigma_{r+2}(\calf) = \frac{C}{n}\cdot\frac{n-r}{r+2}\,\sigma_{r}(\calf).
\end{equation*}
Using induction similarly to Corollary~\ref{C-P1}, yields \eqref{E-cor2}.
\end{proof}

Finally, we will briefly discuss the case of a codimension-one transversally orientable foliati\-on $\calf$ in ${\mathcal D}$ with a unit normal $N$,
 i.e., ${\mathcal D}=T\calf\oplus{\rm span}(N)$, on a sub-Riemannian manifold with a harmonic distribution $\widetilde{\mathcal D}$, see also \cite{r-subR1}.
Put
\[
 {\cal Z}=\nabla^P_N\,N,\quad {\mathcal R}^P_{X}={\mathcal R}^P_{X,N},\quad A=A_N,\quad \sigma_{r+2}=\sigma_{r+2}(N).
\]
 The leafwise divergence of $T_r(A)\ (r>0)$ satisfies, compare with~\eqref{E-divNTk-1},
\begin{equation*}
 \<\Div_\calf T_r(A),\,X\> =
 \sum\nolimits_{\,1\le j\le r}\tr_{\,\calf}\big(T_{r-j}(A)\,{\mathcal R}^P_{\,(-A)^{j-1}X} \big)
\end{equation*}
for any vector field $X$ tangent to $\calf$.
This yields the following corollaries.

\begin{corollary}\label{T-mainLW-1}
For any compact leaf $L$ of a codimension-one foliated sub-Riemannian ma\-nifold, \eqref{E-intCh-C} reads as
\begin{eqnarray*}
 &&\int_{L}\big(
 \<\Div_\calf T_r(A),\,{\cal Z}\>
 -N(\sigma_{r+1}) +\sigma_{1} \sigma_{r+1}\\
 &&\ -\,(r+2)\sigma_{r+2} +\tr_{\,\calf}(T_r(A) {\mathcal R}^P_{N}) + \<T_r(A) {\cal Z}, \,{\cal Z}\> \big)\, {\rm d}\vol_L =0.
\end{eqnarray*}
\end{corollary}

\begin{corollary}\label{T-main01p1}
For a closed sub-Riemannian manifold $(M,{\mathcal D},g)$ with ${\mathcal D}=T\calf\oplus{\rm span}(N)$
and a harmonic distribution $\widetilde{\mathcal D}$, the following integral formula is valid:
\begin{equation*}
 \int_M \big(
 \<\Div_\calf T_r(A),\,{\cal Z}\>
 -(r+2)\,\sigma_{r+2}+\tr_{\,\calf}(T_r(A){\mathcal R}^P_{N})\big)\,{\rm d}\vol_g = 0.
\end{equation*}
\end{corollary}

\begin{example}
\rm
(a) For $r=0$, by Corollary~\ref{T-mainLW-1}, we obtain the integral formula
\begin{equation*}
 \int_{L}\big(\tau_{2}- N(\tau_{1})+\Ric^P_{N,N}
 +\,\<{\cal Z},\, {\cal Z}\> \big)\, {\rm d}\vol_L =0,
\end{equation*}
where $\Ric^P_{N,N}=\tr_{\,\calf}{\mathcal R}^P_{N}=\sum\nolimits_{\,1\le i\le n}\<R^P(e_i,N)N, \,e_i\>$ is the \textit{Ricci $P$-curvature} is the $N$-direction.
Thus, a codimension-one $\calf$ with $\tau_1={\rm const}$ and $\Ric^P_{N,N}>0$ has no compact~leaves.

(b)~For $r=0$, by Corollary~\ref{T-main01p1}, we obtain the following generalization of (\ref{E-sigma2}):
\begin{equation}\label{E-sigma2-new}
 \int_M \big(2\,\sigma_2-\Ric^P_{N,N}
 \big)\,{\rm d}\vol_g=0.
\end{equation}
Let $\dim\calf=1$. The Gaussian $P$-curvature of ${\mathcal D}$ is a function on $M$ defined by $K^P=\<R^P(X,N,N,X)$, where
$X$ is a unit (local) vector tangent to $\calf$.
Then $\sigma_2=0$ and \eqref{E-sigma2-new} reduces to the zero integral of the Gaussian $P$-curvature:
 $\int_M K^P\,{\rm d}\vol_g=0$.
\end{example}

\section{Conclusion}

We suggest that integral formulas \eqref{E-intCh-B}, \eqref{E-intCh-C}, \eqref{E-intCh}, \eqref{E-intNTNT}, and, in particular, \eqref{E-PW2-int}, are a very good tool for understanding the geometry of (sub-)Riemannian manifolds equipped with foliations or distributions.
We delegate the following for further study.

1.
Our integral formulas can be extended for the case of ${\mathcal D}={\mathcal D}_1\oplus{\mathcal D}_2$ -- the sum of two smooth distributions,
see \cite{r-subR0} in relation to \eqref{E-PW2-int} and \cite{r8} for ${\mathcal D}=TM$.
In other words, we take a non-integrable distribution ${\mathcal D}_1$ instead of a  foliation $\calf$ and ${\mathcal D}_2$ instead of ${\mathcal N}\calf$.
This naturally appears when ${\mathcal D}\subset TM$ is a hyperdistribution, whose shape operator has an eigenvalue of constant multiplicity,
and ${\mathcal D}_1$ is the corresponding eigen-distribution.

2. Our integral formulas can be extended for foliations and distributions defined outside of~a~``singularity~set" $\Sigma$
(a finite union of pairwise disjoint closed submanifolds of codimension greater than $2$ of a closed manifold $M$)
under additional assumption of convergence of certain integrals.
Then, instead of the Divergence theorem, we apply the following, see~\cite[Lemma~2]{lw2}: if $X$ is a vector field on $M\setminus\Sigma$ such that $\int_M\|X\|^2\,{\rm d}\vol_g < \infty$, then
$\int_M \Div X\,{\rm d}\vol_g = 0$.

3. Mathematicians have shown a specific interest in manifolds equipped with several distributions, e.g., webs composed of different foliations 
and multiply warped products.
In~\cite{rov-IF-k}, we introduced the mixed scalar curvature for $k>2$ pairwise orthogonal complementary distributions on $(M,g)$
and proved the integral formula similar to \eqref{E-PW-int} and \eqref{E-PW2-int} for this~case.
This study was continued in \cite{RS-1} for the case of a linear connection instead of the Levi-Civita connection.
We~suggest that our integral formulas
can be extended for sub-Riemannian manifolds equipped with several orthogonal foliations
(in ${\mathcal D}$) and certainly defined
${\rm S}^P_{\,\rm mix}$.

4. Finally, one can extend our integral formulas for holomorphic foliations of
complex sub-Riemannian manifolds, see \cite{sve} for the case of ${\mathcal D}=TM$ and \eqref{E-PW-int}.


\end{document}